\documentclass[a4paper,12pt]{article}
\usepackage{amsmath}
\usepackage{amssymb}
\usepackage{amscd}
\usepackage{amsthm}
\usepackage{enumerate}
\ifdefined\directlua
\usepackage{breqn}
\usepackage{fontspec}
\usepackage{microtype}
\usepackage{tikz}
\usetikzlibrary{matrix,arrows,decorations.pathmorphing}
\usepackage{tikz-cd}
\usepackage{geometry}
\usepackage{enumerate}
\usepackage{amsrefs}
\else

\newenvironment{dmath*}{\begin{displaymath}}{\end{displaymath}}

\fi
\newtheorem{theorem}{Theorem}

\newcommand{\cV}{\mathcal V}

\newcommand{\PG}{\mathrm{PG}\,}

\newcommand{\FF}{{\mathbb F}}

\newtheorem{proposition}[theorem]{Proposition}

\def\eqn#1$$#2$${\begin{equation}\label#1#2\end{equation}}

\begin{document}

\author{S.Capparelli and V.Pepe\thanks{The authors acknowledge the support of the project ``Decomposizione, propriet\`{a} estremali di grafi e combinatoria di polinomi ortogonali" of the SBAI Department of Sapienza University of Rome.}}
\title{A non-existence result on symplectic semifield spreads}
\maketitle

\begin{abstract}
We prove that there do not exist non-Desarguesian symplectic semifield spreads of PG$(5,q^2)$, $q\geq 2^{14}$ even, whose associated semifield has center containing $\FF_q$, by proving that the only $\FF_q$-linear set of rank 6 disjoint from the secant variety of the quadric Veronese variety of PG$(5,q^2)$ is a plane with three points of the Veronese surface of PG$(5,q^6)\setminus$PG$(5,q^2)$.
\end{abstract}

\section{Introduction}
Let PG$(n-1,q)$ be the projective space of dimension $n-1$ over the finite field $\FF_q$ of order $q$. Let $M(n,q)$ be the set of $n\times n$ matrices over $\FF_q$.

A {\it planar spread} $\mathcal{S}$ of PG$(2n-1,q)$, which we will call simply {\it spread} from now on, is a partition of the point-set in $(n-1)$-dimensional subspaces. With any spread $\mathcal{S}$ it is associated a translation plane $A(\mathcal{S})$ of order $q^n$ in the following way: embed PG$(2n-1,q)$ in PG$(2n,q)$ as a hyperplane section, then the points of $A(\mathcal{S})$ are the points of PG$(2n,q)\setminus\mathrm{PG}(2n-1,q)$, the lines are the $n$-dimensional subspaces of PG$(2n,q)$ intersecting PG$(2n-1,q)$ in an element of $\mathcal{S}$ and the incidence is containment (see e.g. \cite[Section 5.1]{demb}). Translation planes associated with different spreads are isomorphic if and only if there is a collineation of PG$(2n-1,q)$ mapping one spread to the other (see \cite{andre} or \cite[Chapter 1]{luneburg}). Without loss of generality, we may always assume that $S(\infty):=\{(\mathbf{0},\mathbf{y}), \mathbf{y}\in \FF_q^n\}$ and $S(0):=\{(\mathbf{x},\mathbf{0}), \mathbf{x}\in \FF_q^n\}$ belong to $\mathcal{S}$, hence we may write $\mathcal{S}=\{S(A), A \in \mathbb{C}\}\cup S(\infty)$, with $S(A):=\{(\mathbf{x},\mathbf{x}A),\mathbf{x}\in \FF_q^n \}$ and $\mathbb{C} \subset M(n,\FF_q)$ such that $|\mathbb{C}|=q^n$, $\mathbb{C}$ contains the zero matrix and  $A-B$ is non-singular for every $A,B \in \mathbb{C}$. The set $\mathbb{C}$ is called the spread set associated with $\mathcal{S}$.

A spread $\mathcal{S}$ is said to be \textit{Desarguesian} if $A(\mathcal{S})$ is isomorphic to AG$(2,q^n)$ and hence a plane coordinatized by the field of order $q^n$. The spread $\mathcal{S}$ is said to be a \textit{semifield spread} if $A(\mathcal{S})$ is a plane of Lenz-Barlotti class V and this is equivalent to saying that $A(\mathcal{S})$ is coordinatized by a semifield. A finite {\it semifield} $(F,+,*)$ is a finite nonassociative division algebra. If $(F,+,*)$ satisfies all the axioms for a semifield except, possibly, the existence of an identity element for the multiplication, then it is called a {\it presemifield}. A semifield spread $\mathcal{S}$ is such that there exists an elementary abelian subgroup $G$ of $PGL(2n,q)$ of order $q^n$ fixing an element $X \in \mathcal{S}$ point-wise and acting regularly on $\mathcal{S}\setminus X$.
If we set $X=S(\infty)$, then $\mathbb{C}$ turns out to be closed under addition; hence $\mathbb{C}$ is a set of $q^n$ $n\times n$ invertible matrices over $\FF_q$ that form a subgroup of the additive group of $M(n,\FF_q)$ (\cite[Section 5.1]{demb}). Then $\mathbb{C}$ is a vector space over some subfield of $\FF_q$. This has led to the following geometric interpretation (see \cite{lunardon2003} as it first appeared for $n=2$ and \cite{lavrauw2011} for the general case). Let PG$(n^2-1,q)$ be the projective space induced by $M(n,\FF_q)$ and let $\mathcal{D}$ be the algebraic variety of PG$(n^2-1,q)$ consisting of the singular matrices; $\mathcal{D}$ is a so-called determinantal variety and it is the $(n-2)$-th secant variety of the Segre variety $\Sigma_{n-1,n-1}:=\mathrm{PG}(n-1,q)\times\mathrm{PG}(n-1,q)$ (see \cite[Ch.9]{Harris}). If $\mathbb{C}$ is an $\FF_s$-vector space, $q=s^t$, then $\dim_{\FF_s}\mathbb{C}= nt$ and it defines a subset $\Lambda$ of PG$(n^2-1,q)$ called $\FF_s$-linear set of rank $nt$ (for a beautiful and complete overview of the topic see \cite{linear_set}). So finding a semifield spread of PG$(2n-1,q)$ (and hence a semifield plane of order $q^n$) is equivalent to finding an $\FF_s$-linear set of PG$(n^2-1,q)$, $q=s^t$, of rank $nt$ disjoint from $\mathcal{D}$. We recall that the left nucleus $N_l$ of a semifield $F$ is the set $\{k \in F| k*(x*y)=(k*x)*y, \forall x,y \in F\}$ and the center is the set $ \mathcal{K}=\{k \in F| k*(x*y)=(k*x)*y, x*(k*y)=(k*x)*y, x*(y*k)=(x*y)*k\quad \forall x,y \in F\}$, so an $\FF_s$-linear set of PG$(n^2-1,q)$ of rank $nt$ disjoint from $\mathcal{D}$, with $\FF_s$ maximum subfield of linearity leads to a semifield of order $q^n$, center isomorphic to $\FF_s$ and left nucleus isomorphic to $\FF_q$. In \cite{lavrauw2011}, the following more general result has been proved: two semifields of order $q^n$, with left nucleus containing $\FF_q$ and center containing $\FF_s$, are isotopic (that is they coordinatize isomorphic translation planes) if and only if there is collineation fixing the two systems of maximal subspaces of $\Sigma_{n-1,n-1} \subset \mathrm{PG}(n^2-1,q)$ mapping one $\mathbb{F}_s$-linear set to the other.

The spread $\mathcal{S}$ is said to be \textit{symplectic} if the elements of $\mathcal{S}$ are totally isotropic with respect a \textit{symplectic polarity} of PG$(2n-1,q)$.
Let $\beta((\mathbf{x}_1,\mathbf{y}_1),(\mathbf{x}_2\mathbf{y}_2))=\mathbf{x}_1\mathbf{y}_2^T-\mathbf{y}_1\mathbf{x}_2^T$, $\mathbf{x}_i,\mathbf{y}_i \in \FF_q^n$, then $\beta$ is an alternating bilinear form of PG$(2n-1,q)$ and it induces a symplectic polarity $\perp$. The subspace $S(\infty)$ is clearly totally isotropic with respect to $\perp$. The subspace $S(A)\in \mathcal{S}$ is totally isotropic if and only if $\mathbf{x}_1(\mathbf{x}_2A)^T-(\mathbf{x}_1A)\mathbf{x}_2^T=0 \quad \forall \mathbf{x}_1,\mathbf{x}_2 \in \FF_q^n$ and this is true if and only if $A$ is symmetric. Hence to a symplectic spread $\mathcal{S}$ it is possible to associate a spread set $\mathbb{C}$ consisting of symmetric matrices (see e.g. \cite{kantor2003,maschietti2003}). The symmetric matrices form a subspace of $M(n,\FF_q)$ of dimension $\frac{n(n+1)}{2}$ that then induces a subspace PG$(\frac{n(n+1)}{2}-1,q)$ of $PG(n^2-1,q)$. The rank 1-symmetric matrices are the Veronese variety $\cV$ of degree 2 of PG$(\frac{n(n+1)}{2}-1,q)$ (this is the so called determinantal representation of the Veronese variety of degree 2, see \cite[Example 2.6]{Harris}). Hence the singular symmetric matrices form the $(n-2)$-th secant variety, say $\cV_{n-2}$ of the Veronese variety. So to a symplectic semifield spread of PG$(2n-1,q)$ there corresponds an $\FF_s$-linear set $\Lambda$, $q=s^t$, of PG$(\frac{n(n+1)}{2}-1,q)$ of rank $tn$ such that $\Lambda \cap \cV_{n-2}=\emptyset$ (see also \cite{lmpt2011}).

For odd $q$, there are many examples of non-Desarguesian symplectic semifield spreads of  PG$(2n-1,q)$, also due to the connection between DO polynomials and commutative semifields of odd order (by \cite{kantor2003}, any symplectic semifield defines a presemifield isotopic to a commutative semifield).

When $q$ is even, a symplectic spread of PG$(2n-1,q)$ gives rise to many interesting geometric structures.
For even $q$, the symplectic polar space of PG$(2n-1,q)$ is isomorphic to the orthogonal space of PG$(2n,q)$ induced by a non-singular parabolic quadric $\mathcal{Q}(2n,q)$ as, for even $q$, $\mathcal{Q}(2n,q)$ has a nucleus $N$ through which all the tangent lines pass: projecting the points of $\mathcal{Q}(2n,q)$ from $N$ to a hyperplane, we obtain an incidence structure that is a symplectic polar space. Hence a symplectic spread $\mathcal{S}$ of PG$(2n-1,q)$ is equivalent to a spread of $\mathcal{Q}(2n,q)$, i.e., a partition of the point-set of $\mathcal{Q}(2n,q)$ in subspaces of maximum dimension $n-1$. If $n$ is odd, a spread of $\mathcal{Q}(2n,q)$ in turn leads to a spread of a hyperbolic quadric $\mathcal{Q}^+(2n+1,q)$: embed $\mathcal{Q}(2n,q)$ in $\mathcal{Q}^+(2n+1,q)$ as a hyperplane section, then a maximal subspace of $\mathcal{Q}(2n,q)$ is contained in two maximal subspaces of $\mathcal{Q}^+(2n+1,q)$, one for each system, so if we pick one, we obtain a spread of $\mathcal{Q}^+(2n+1,q)$, i.e., a partition of the point-set of $\mathcal{Q}^+(2n+1,q)$ in subspaces of dimension $n$. By a suitable choice of coordinates, every such a spread gives rise to a Kerdock set, i.e., a set of $q^n$ skew-symmetric $(n+1)\times (n+1)$ matrices, $n$ odd, such that the difference of any two of them is nonsingular. From such a Kerdock set it is possible to obtain a remarkable class of codes: namely, binary Kerdock codes of length $2^{n+1}$ (see \cite{ccks1997}). By slicing a spread of $\mathcal{Q}^+(2n+1,q)$ in nonequivalent hyperplanes under the action of the subgroup of the orthogonal group fixing it, we obtain nonequivalent spreads of $\mathcal{Q}(2n,q)$, called \textit{cousins} in \cite{kantorI}, and hence, for even $q$, nonequivalent symplectic spreads and translation planes. Hence starting from one symplectic spread it is possible to get many more (see \cite{kantorI}).

In the particular case when $n=3$, we obtain a spread of $\mathcal{Q}^+(7,q)$ and hence, by triality, we obtain an ovoid of $\mathcal{Q}^+(7,q)$, a combinatoric structure that gives rise to many others (see \cite{kantorI,kantor1982}).

In this article, we are focused on symplectic semifield spreads of PG$(5,q)$, when $q$ is even. In such a case, only two nonsporadic examples are known: the Desarguesian spread and one of its cousin (see \cite{kantorI}), so they are both obtained by slicing the so called Desarguesian spread of
$\mathcal{Q}^+(7,q)$. In the first case, the associated translation plane is the Desarguesian plane, hence it is coordinatized by the finite field of order $q^3$ and the relevant linear
set is actually linear on $\FF_q$. In the second case we have, somehow, the ``opposite situation'': the semifield spread is associated to a spread set $\mathbb{C}$ that gives rise to an
$\mathbb{F}_2$-linear set $\Lambda$ of PG$(5,q)$, where $\FF_2$ is the maximum subfield of $\FF_q$ for which $\Lambda$ is linear, hence the associate semifield has order $q^3$ and center $\FF_2$. In this article we prove the following:

 \textbf{Main result.}
 \textit{The only $\FF_q$-linear set $\Lambda$ of rank $6$ disjoint from the secant variety of the Veronese surface of PG$(5,q^2)$, for even $q\geq 2^{14}$, is a plane with three  points of the Veronese surface of PG$(5,q^6)\setminus$PG$(5,q^2)$, and hence there do not exist non-Desarguesian semifield symplectic spreads of PG$(5,q^2)$, whose associated semifield has center containing $\FF_q$}.

\section{Quadric Veronesean and its secant variety}
In this section we denote by $\mathbb{P}^{n-1}$ the $(n-1)$-dimensional projective space over the generic field $\mathbb{F}$.

The Veronese map of degree $2$
 $$v_2:(x_0,x_1,\ldots,x_{n-1}) \in \mathbb{P}^n \longmapsto (\ldots, \mathbf{x}^{l},\ldots)\in \mathbb{P}^{N-1}$$ is such that $\mathbf{x}^l$ ranges over all
 monomials of degree $2$ in $x_0,x_1,\ldots,x_{n-1}$, hence $N=\frac{n(n+1)}{2}$. The image of $v_2$ is an algebraic variety called quadric Veronese variety. If we use the so-called
  determinantal representation of the Veronese variety of degree 2 (see \cite[Example 2.6]{Harris}), then $\mathbb{P}^{N-1}$ is induced by the subspace of $M(n,\mathbb{F})$ consisting of
   symmetric matrices and, by $v_2(x_0,x_1,\ldots, x_{n-1})=A$ such that $a_{ij}=x_ix_j$, the Veronese variety turns out to be the intersection of $\Sigma_{n-1,n-1}=\mathbb{P}^{n-1} \times \mathbb{P}^{n-1}$ with such a $\mathbb{P}^{N-1}$.
   We recall that the Segre variety $\Sigma_{n-1,n-1}\subset \mathbb{P}^{n^2-1}$ consists of all rank 1 matrices of $M(n,\mathbb{F})$ and the $k$th secant variety of $\Sigma_{n-1,n-1}$,
    i.e., the union of $k$-subspaces spanned by points of $\Sigma_{n-1,n-1}$, is the algebraic variety consisting of the matrices of rank at most $k+1$ (see \cite[Example 9.2]{Harris}),
    hence  the $k$th secant variety of the Veronese variety consists of the symmetric matrices of rank at most $k+1$.

 We are interested in the particular case when $n=3$ and hence the quadric Veronese surface is $\cV=v_2(\mathbb{P}^2)=\left\{\left(
                                                                                               \begin{array}{ccc}
                                                                                                 x^2 & xy & xz \\
                                                                                                 xy & y^2 & yz \\
                                                                                                  xz & yz & z^2 \\
                                                                                               \end{array}
                                                                                             \right),x,y,z \in \mathbb{F}\right\}
  \subset \mathbb{P}^5$. The image of a line $\ell$ of $\mathbb{P}^2$ is a conic, intersection of $\cV$ with a suitable plane. If the characteristic of $\mathbb{F}$ is even, then a line of
   the plane is tangent to a conic if and only if it contains a fixed point, called the nucleus of the conic. Then, when the field characteristic is even, there exists a plane $\pi_{N}$ of
   $\mathbb{P}^5$, called the nucleus plane, such that the nucleus of each conic $v_2(\ell)$ belongs to $\pi_{N}$ and vice versa each point of $\pi_{N}$ is a nucleus of one and only one conic
   $v_2(\ell)$ (see \cite[Th.25.1.17]{GGG} for this particular case and \cite{thas_vanmaldeghem2004} for the general case). In the representation we have chosen, we have

  $$\pi_N=\left\{\left(
   \begin{array}{ccc}
    0 & t & u \\
    t & 0 & v \\
    u & v & 0 \\
    \end{array}
    \right)
   t,u,v \in \mathbb{F}\right\}.$$

    The automorphism group $\hat{G}$ of $\cV$ is the lifting of $G=PGL(3,\mathbb{F})$ acting in the obvious way: $v_2(P)^{\hat{g}}=v_2(P^g)$, $\forall g \in PGL(3,\mathbb{F})$.

 The secant variety $\cV_1$ of $\cV$ is the variety of $\mathbb{P}^5$ consisting of the points lying on a line secant to $\cV$. It is a determinantal variety, i.e., it consists of the symmetric $3 \times 3$ matrices $\left\{\left(
                                                                                           \begin{array}{ccc}
                                                                                             t_0 & t_3 & t_4 \\
                                                                                             t_3 & t_1 & t_5 \\
                                                                                             t_4 & t_5 & t_2 \\
                                                                                           \end{array}
                                                                                         \right),t_i \in \mathbb{F},i=0,1,\ldots,5\right\}
 $ with zero determinant. So $\cV_1$ is a hypersurface of $\mathbb{P}^5$ with equation $t_0t_1t_2-t_0t_5^2-t_1t_4^2-t_2t_3^2+2t_3t_4t_5=0$. The automorphism group $\hat{G}$ of $\cV$ obviously fixes $\cV_1$. It is well known that the singular points of $\cV_1$ are the points of $\cV$ (see, e.g., \cite[Excercise 14.15]{Harris}), that are actually double points. Moreover, it is easy to see that, when the field characteristic is even, every tangent hyperplane contains $\pi_N$ and viceversa. We observe that $\pi_N\subset \cV_1$.

\begin{proposition}\label{planes}
 Let $\cV_1$ be the secant variety of the quadric Veronese variety $\cV=v_2(\mathbb{P}^2)\subset \mathbb{P}^5$. The maximal linear subspaces contained in $\cV_1$ are planes and they are: the planes intersecting $v_2(\mathbb{P}^2)$ in $v_2(\ell)$, with $\ell$ a line of $\mathbb{P}^2$, the planes tangent to the Veronese variety and, if the field characteristic is even, the nucleus plane.
\end{proposition}
\begin{proof}
 The maximal linear subspaces contained in the determinantal variety of the $3\times 3$ singular matrices are of four types: the matrices with a common kernel, the matrices with a common image, the matrices $A$ such that $A(V)\subset W$, with $V$ and $W$ subspaces of $\mathbb{F}^3$ of dimension 2 and 1 respectively (here by dimension we mean vector space dimension), and the skew-symmetric matrices (with zero diagonal) (see \cite[page 113]{Harris}). The variety $\cV_1$ is the variety of the symmetric $3\times 3$ singular matrices and it is easy to see that we can get subspaces of projective dimension at most 2. Moreover, the matrices of the first two types coincide. Let $L$ be the linear subspace of the symmetric singular matrices $A$ such that $(a,b,c)A=\mathbf{0}$ for some fixed non-zero vector $(a,b,c) \in \mathbb{F}^3$. Consider the line $\ell:=\{(x,y,z)\in \mathbb{P}^2|ax+by+cz=0\}$, then $v_2(\ell)$ is clearly contained in $L$. Let now $L$ be linear subspace of the symmetric singular matrices $A$ such that $A(V)\subset W$, with $V$ and $W$ subspaces of $\mathbb{F}^3$ of dimension 2 and 1 respectively. Let $w_1,w_2 \in \mathbb{F}^3$ two vectors defining $W$, i.e. $w_i\cdot w=0 \forall w \in W$, and let $v_1,v_2$ be two independent vectors of $V$, then we have $L=\{A \text{ singular and symmetric}|v_iA(w_j)^T=0 \forall i=1,2\}$. In order to have a linear space of dimension 3, we must have $\{v_1,v_2\}=\{w_1,w_2\}$, hence $L=\{A \text{ singular and symmetric}|v_iA(v_j)^T=0, i=1,2\}$. Let $v_i=(a_i,b_i,c_i),i=1,2$ and let $A \in \cV$, then $v_iA(v_j)^T=(a_ix+b_iy+c_iz)(a_jx+b_jy+c_jz)$. Hence $L\cap \cV$ is the unique point $(x,y,z) \in \mathbb{P}^2$ such that $(a_1x+b_1y+c_1z)=(a_2x+b_2y+c_2z)=0$ and it is the tangent plane to $\cV$ at $v_2(x,y,z)$. The skew-symmetric matrices (with zero diagonal) belong to $\cV_1$ if and only if the field characteristic is even and they form the nucleus plane $\pi_N$.
\end{proof}

In the following, we will refer to the planes containing $v_2(\ell)$, with $\ell$ a line of $\mathbb{P}^2$, as \textit{conic planes} and to the planes tangent to $\cV$ as \textit{tangent planes}.

 \begin{proposition}\cite[Theorem 25.1.17]{GGG}\label{intersection}
 Let the field characteristic be even. Each conic plane intersects the nucleus plane in a point and each tangent plane intersects the nucleus plane in a line.
\end{proposition}

 \begin{proposition}\label{empty}
Let $\cV_1$ be the secant variety of the Veronese surface of PG$(5,q)$, $q\geq 2^8$, then a plane $\pi$ is disjoint from $\cV_1$ if and only if $\pi \cap \cV_1$ consists of three lines through three points of the Veronese surface over $\FF_{q^3}$.
\end{proposition}
\begin{proof}
Over the algebraic closure of $\FF_q$, $\pi \cap \cV_1$ consists of a curve $\mathcal{C}$ of degree 3. By \cite[Corollary 7.4]{cafure_matera}, if the curve of degree $d$ is absolutely irreducible and $q>\max\{4\cdot d^2,2\cdot d^4\}$, then it has at least one $\FF_q$ rational point. Hence, for $q\geq 2^8$, $\mathcal{C}$ has no $\FF_q$-rational points if and only if it is reducible and the only possibility is that $\mathcal{C}$ consists of three non--concurrent lines over $\FF_{q^3}$, say $\ell,\ell^{\sigma},\ell^{\sigma^2}$, where $\sigma$ is the $\FF_{q}$--linear collineation of the plane induced by $Gal(\FF_{q^3}/\FF_q)$. Let $P$ be $\ell \cap \ell^{\sigma}$, hence the points $P,P^{\sigma},P^{\sigma^2}$ are singular for $\mathcal{C}$. Suppose that $P$ is not a singular point for $\cV_1$, then $\pi$ is contained in the tangent hyperplane of $P$. Since $q$ is even, the tangent hyperplane contains also $\pi_N$, hence $\pi$ and $\pi_N$ intersects in at least one point and since the two planes are defined over $\FF_q$, the point is $\FF_q$-rational. As $\pi_N \subset \cV_1$, we get a an $\FF_q$-rational point of $\pi\cap \cV_1$. Hence $P,P^{\sigma},P^{\sigma^2}$ have to be singular for $\cV_1$, i.e. $P,P^{\sigma}$ and $P^{\sigma^2} \in \cV$ and they are $\FF_{q^3}\setminus \FF_q$-rational.
\end{proof}

\section{Proof of the main result}

An $\FF_{q}$--linear set $\Lambda$ of rank $m$ of PG$(n-1,q^2)$ is a set of points of PG$(n-1,q^2)$ defined by an $m$--dimensional vector space over $\FF_q$. Let $P=(x_0,x_1,\ldots,x_{n-1})\in\Lambda$, then, by definition, $\lambda(x_0,x_1,\ldots,x_{n-1})\in\Lambda$ $\forall \lambda \in \FF_{q}$. If $\lambda(x_0,x_1,\ldots,x_{n-1})\in\Lambda$ $\forall \lambda \in \FF_{q^2}$, then we say that $P$ has \textit{weight 2}, otherwise $P$ is said to be of \textit{weight 1}. If each point of $\Lambda$ has weight 1, then $\Lambda \cong$PG$(m-1,q)$. If $\Lambda$ contains points of weight 2 and $W$ is an $m$--dimensional vector space defining $\Lambda$, then we have $W=W_1\oplus W_2$, with $W_2$ the maximal subspace of $W$ that is a vector space also over $\FF_{q^2}$, $\dim_{\FF_q}W_1=h$, $\dim_{\FF_{q}}W_2=2k$, $h+2k=m$, $\langle \Lambda \rangle$ is a PG$(h+k-1,q^2)$ and $\Lambda$ is a cone with vertex a PG$(k-1,q^2)$ and base a PG$(h-1,q)$.

Throughout this section, we let $q$ be even. If $f_i \in \mathbb{F}[x_0,x_1,\ldots,x_{n-1}]$, $i \in \mathcal{I}$, then let $V(f_i,i\in \mathcal{I}) \subset \mathbb{P}^{n-1}$ be the algebraic variety consisting of the solutions of $f_i=0$ $\forall i \in \mathcal{I}$.

\begin{theorem}
There does not exist an $\mathbb{F}_q$-linear set $\Lambda$ of rank 6 disjoint from $\cV_1$ in PG$(5,q^2)$ for $q \geq 2^{14}$, unless $\Lambda$ is a plane of PG$(5,q^2)$.
\end{theorem}
\begin{proof}
The variety $\cV_1$ is a hypersurface of degree 3 consisting of the zeros of $t_0t_1t_2+t_0t_5^2+t_1t_4^2+t_2t_3^2$.
 The nucleus plane $\pi_N:=\{(0,0,0,t_3,t_4,t_5),t_i \in \FF_{q^2},i=3,4,5\}$ is contained in $\cV_1$ so we must have $\Lambda \cap \pi_N=\emptyset$ in $\PG(5,q^2)$ and hence
 $\Lambda= \{(x,y,z,F_1,F_2,F_3),x,y,z \in \FF_{q^2}\}$, with $F_i$ $\FF_q$-linear function of $x,y,z$. For $a \in \FF_q$, let $t^2+a t+1$ be  an irreducible polynomial  over $\FF_q$ and let $\xi$
 be one of its roots in $\FF_{q^2}$. Then $\{1,\xi\}$ is a basis of $\FF_{q^2}$ as $\FF_q$-vector space. Hence we can write $x=x_1+x_2\xi$, $y=y_1+y_2\xi$, $z=z_1+z_2\xi$, with
 $x_i,y_i,z_i \in \FF_q,i=1,2$ and $F_i:\FF_{q^2}^3\longrightarrow \FF_{q^2}$ as $F_1(x_1,x_2,y_1,y_2,z_1,z_2) =l_1+\xi l_2$, $F_2(x_1,x_2,y_1,y_2,z_1,z_2) =m_1+\xi m_2$,
 $F_3(x_1,x_2,y_1,y_2,z_1,z_2) =n_1+\xi n_2$, with $l_i,m_i,n_i$ linear functions of $x_1,x_2,y_1,y_2,z_1,z_2$. In order to avoid confusion, we will denote by $\Sigma$ the PG$(5,q^2)$ where we have defined $\cV_1$ and by $\Sigma(\FF_{q^{2n}})$ the geometry obtained by the field extension of degree $n$ of $\FF_{q^2}$, whereas we will denote by $\Pi$ the PG$(5,q)=\{(x_1,x_2,y_1,y_2,z_1,z_2,l_1,l_2,m_1,m_2,n_1,n_2),x_i,y_i,z_i \in \FF_{q}\}$ and by $\Pi(\FF_{q^n})$ the projective space containing $\Pi$ as subgeometry obtained by the field extension of degree $n$ of $\FF_{q}$. We remark that $\Pi$ induces the linear set $\Lambda$ in $\Sigma$.

A point in $\Lambda \cap \cV_1$ has to fulfill:

\noindent $(x_1+x_2\xi)(y_1+y_2\xi)(z_1+z_2\xi)+(x_1+x_2\xi)(n_1+\xi n_2)^2+(y_1+y_2\xi)(m_1+\xi m_2)^2+(z_1+z_2\xi)(l_1+\xi l_2)^2= (x_1+x_2\xi)(y_1+y_2\xi)(z_1+z_2\xi)+(x_1+x_2\xi)(n_1^2+\xi^2n_2^2)+(y_1+y_2\xi)(m_1^2+\xi^2m_2^2)+(z_1+z_2\xi)(l_1^2+\xi^2l_2^2)=$

\noindent$x_1y_1z_1+\xi(x_1y_1z_2+x_1y_2z_1+x_2y_1z_1)+\xi^2(x_1y_2z_2+x_2y_1z_2+x_2y_2z_1)+\xi^3x_2y_2z_2+x_1n_1^2+y_1m_1^2+z_1l_1^2+i(x_2n_1^2+y_2m_1^2+z_2l_1^2)+\xi^2(x_1n_2^2+y_1m_2^2+z_1l_2^2)+\xi^3(x_2n_2^2+y_2m_2^2+z_2l_2^2)=0.$

As $\xi$ is a root of $t^2+a t +1$, we have $\xi^2=a \xi +1$ and $\xi^3=a \xi^2+\xi= a(a \xi+1)+\xi=(a^2+1)\xi+a$, so:

\noindent$x_1y_1z_1+\xi(x_1y_1z_2+x_1y_2z_1+x_2y_1z_1)+\xi^2(x_1y_2z_2+x_2y_1z_2+x_2y_2z_1)+\xi^3x_2y_2z_2+x_1n_1^2+y_1m_1^2+z_1l_1^2+\xi(x_2n_1^2+y_2m_1^2+z_2l_1^2)+\xi^2(x_1n_2^2+y_1m_2^2+z_1l_2^2)+\xi^3(x_2n_2^2+y_2m_2^2+z_2l_2^2)=$

\noindent $x_1y_1z_1+x_1y_2z_2+x_2y_1z_2+x_2y_2z_1+a x_2y_2z_2+x_1n_1^2+y_1m_1^2+z_1l_1^2+ x_1n_2^2+y_1m_2^2+z_1l_2^2+a (x_2n_2^2+y_2m_2^2+z_2l_2^2)+\xi(x_1y_1z_2+x_1y_2z_1+x_2y_1z_1+x_2n_1^2+y_2m_1^2+z_2l_1^2+a(x_1y_2z_2+x_2y_1z_2+x_2y_2z_1+x_1n_2^2+y_1m_2^2+z_1l_2^2)+(a^2+1)(x_2y_2z_2+x_2n_2^2+y_2m_2^2+z_2l_2^2))=0,$

\noindent implying:

$$\begin{cases}
f:=x_1y_1z_1+x_1y_2z_2+x_2y_1z_2+x_2y_2z_1+a x_2y_2z_2+x_1n_1^2+y_1m_1^2+z_1l_1^2\\+ x_1n_2^2+
y_1m_2^2+z_1l_2^2+a (x_2n_2^2+y_2m_2^2+z_2l_2^2)=0\\
g:=x_1y_1z_2+x_1y_2z_1+x_2y_1z_1+x_2n_1^2+y_2m_1^2+z_2l_1^2+a(x_1y_2z_2+x_2y_1z_2\\+x_2y_2z_1+
x_1n_2^2+y_1m_2^2+z_1l_2^2)+(a^2+1)(x_2y_2z_2+x_2n_2^2+y_2m_2^2+z_2l_2^2)=0
\end{cases}.
$$
\bigskip
\noindent and the only solution for this system of equations must be $x_i,y_i,z_i=0$ $ \forall i=1,2$.

That is equivalent to asking that the algebraic variety $V(f,g)$ defined in PG$(5,q)$ (i.e., $f$ and $g$ polynomials with coefficients in $\FF_q$) should not contain $\mathbb{F}_q$--rational points. Classical results, as the Lang-Weil bound (\cite{LangWeil}), state that when $V$ is an absolutely irreducible variety of dimension $r$ defined by polynomials over the finite field $\FF_q$, for $q$ "big enough", $|V|$ is roughly $q^r$. More precisely, by \cite[Corollary 7.4]{cafure_matera}, if $d$ is the degree of $V$, then for $q> \max\{2(r+1)d^2,2d^4\}$ $V$ has at least one $\FF_q$-rational point. In our case $r\leq4$ and $d\leq 9$, so for $q>2\cdot 9^4$, if $V(f,g)$ is absolutely irreducible, it has at least one point.
Hence, assuming that $q\geq 2^{14}$, in order to get a variety with no $\FF_q$--rational points, we must have a variety $V$ which is reducible over some extension of $\FF_q$, say $\FF_{q^t}$, and such that
$V=W \cup W^{\sigma} \cup \cdots \cup W^{\sigma^{t-1}}$, $W$ a (possible reducible) variety with the same dimension as $V$, $\langle \sigma \rangle =Gal( \FF_{q^t}/\FF_q)$ and $W \cap W^{\sigma} \cap \cdots \cap W^{\sigma^{t-1}}=\emptyset$, where, by abuse of notation, we denote by $\sigma^i$ also the map $(x_0,x_1,\ldots,x_{n-1}) \in \mathrm{PG}(n-1,q^t)\mapsto (x_0^{\sigma^i},x_1^{\sigma^i},\ldots,x_{n-1}^{\sigma^i})\in\mathrm{PG}(n-1,q^t)$, hence $P=(x_0,x_1,\ldots,x_{n-1}) \in \mathrm{PG}(n-1,q)$ if and only if $P^{\sigma}=P$. We remark that $\sigma^i$ induces an automorphism on the projective space, hence $W$ and $W^{\sigma^i}$ have the same degree and dimension.

 Let $V:=V(f,g)$. Suppose that one the two polynomials, say $f$, is reducible over $\FF_q$, hence $f$ has a linear factor $f_1$ over $\FF_q$ and $V(f_1,g) \subseteq V(f,g)$. As $V$ does not contain $\FF_q$--rational points, $f_1$ cannot be a factor of $g$ and hence the variety $V(f_1,g)$ is a cubic hypersurface of PG$(4,q)$. By the Chevalley–Warning theorem, $V(f_1,g)$ has at least one point over $\FF_q$. Hence both $f$ and $g$ are irreducible over $\FF_q$. In the algebraic closure, $\dim V=3$ unless both $V(f)$ and $V(g)$ are reducible varieties over some field extension and they have some common component. The hypersurface
 $V(f)$ is reducible if and only if $f$ is reducible on some field extension and since all the components have the same degree, the only possibility is that $f=f_1f_2f_3$, where $f_i$ is
 linear and defined over $\FF_{q^3}$. The same is true for $V(g)$, hence $g=g_1g_2g_3$ and, in order to get $\dim V=4$, we must have $f_i=g_j$ for some $i$ and $j$, but then $f=g$ and this is
 obviously not the case.

So we may assume that $\dim V=3$ and that $V=W \cup W^{\sigma} \cup \cdots \cup W^{\sigma^{t-1}}$. We have that $\dim W=3$ and let $\deg W$ be $d$ and the multiplicity be $\mu$ (if $W$ is reducible, then $d$ and $\mu$ are the sum of the degrees and multiplicities of the irreducible components of $W$ of dimension $3$), so $d\mu t=\deg (f)\cdot deg (g)=9$ (see \cite[Theorem 18.4]{Harris}). If $t=9$, then $d=1$ and $W$ is a 3--dimensional subspace over $\FF_{q^9}$. Hence $W$ induces a $\FF_{q^9}$--linear set of rank 4 of $\Sigma(\FF_{q^{18}})$. Such a linear set can be a PG$(3,q^9)$, a cone with vertex a point and base a subline PG$(1,q^9)$ (hence contained in a plane of $\Sigma(\FF_{q^{18}})$), or a line of $\Sigma(\FF_{q^{18}})$. In the first case, we would get a PG$(3,q^9)$ contained in $\cV_1$ in $\Sigma(\FF_{q^{18}})$, but this implies that there is a PG$(3,q^{18})$ contained in $\cV_1$, a contradiction to Proposition \ref{planes}. Suppose that $W$ induces a cone. If a plane PG$(2,q^{18})$ shares with $\cV_1$ $q^{9}+1$ lines, then it is contained in $\cV_1$.  Since the $W^{\sigma^i}$ pairwise intersect in at least a line, the linear sets induced by them intersect in a $\FF_{q^9}$--linear set of rank 2. A $\FF_{q^9}$--linear set of rank 2 of PG$(2,q^{18})$ is either a point or a subline PG$(1,q^9)$. As the $W^{\sigma^i}$ are contained in planes that are all conic or all tangent and the planes of the same type pairwise intersect in exactly a point (see, e.g., \cite[Theorems 25.1.11 and 25.1.16]{GGG}), the linear sets induced by the $W^{\sigma^i}$ pairwise intersect in a point of weight 2. As in a cone there is only one point of weight 2, they all intersect in the same point, contradicting $W \cap W^{\sigma} \cap \cdots \cap W^{\sigma^{t-1}}=\emptyset$. So assume that $W$ induces a line. Again, these lines pairwise intersect in a point, but they are not through the same one, hence we get 9 lines in a plane. If a plane shares with $\cV_1$ 9 lines, then it is contained in $\cV_1$. Let $\pi$ be such a plane, then $\Pi(\FF_{q^9})$ induces $\pi$ in $\Sigma(\FF_{q^{18}})$ and so we cannot have that the linear set induced by $\Pi(\FF_{q^9})$ intersects $\cV_1$ in only 9 lines of $\Sigma(\FF_{q^{18}})$. Let $t$ be 3. If $\mu=3$, then $d=1$ and $W$ is a 3--dimensional subspace. Reasoning as before, we get that $\Pi(\FF_{q^3})$ induces a plane $\pi$ of $\Sigma(\FF_{q^3})$ intersecting $\cV_1$ in $\Sigma(\FF_{q^3})$ in three lines not through the same point, hence $\Pi$ induces a subplane PG$(2,q^2)$ disjoint from $\cV_1$ in $\Sigma$, as shown in Proposition \ref{empty}. Let now $\mu=1$.  Hence $d=3$. Suppose that $W$ is reducible, then there must be a component of $W$ of dimension $3$ and degree $1$, getting again the previous case.

Suppose now that $d=3=t$, $\mu=1$ and $W$ is irreducible.

 If $W$ is a variety of PG$(5,q^3)$, then it is a variety of minimal degree (see \cite[Corollary 18.12]{Harris}) and the only possibility is that $W$ is the Segre variety $\Sigma_{1,2}$,
 product of a line and a plane. The variety $\Sigma_{1,2}$ contains two ruling of maximal subspaces, one consisting of lines and the other of planes.
 A plane of $\Pi(\FF_{q^3})$ induces a $\FF_{q^3}$--linear set of rank 3 of $\Sigma(\FF_{q^6})$. Such a linear set is either a subplane PG$(2,q^3)$ or the points a line PG$(1,q^6)$, with only one of them with weight 2. Suppose that one of them, say $\pi$, induces a subplane. If a subplane PG$(2,q^3)$ is contained in $\cV_1$ in $\Sigma(\FF_{q^6})$, then all the PG$(2,q^6)$ containing it is contained in $\cV_1$ and, as $q$ is even, such a plane contains at least a point of $\pi_N$ by Proposition \ref{intersection}. As $\pi_N$ is setwise fixed by the automorphism of $\Sigma(\FF_{q^6})$ induced by $Gal(\FF_{q^6}/\FF_{q^3})$, we have at least an intersection point in the subplane induced by $\pi$. Hence there exists $P\in \pi$ such that $P$ induces a point of $\pi_N$. The subspace $\langle p, P^{\sigma}$, $P^{\sigma^2}\rangle \subset \Pi(\FF_{q^3})$ is setwise fixed by $Gal(\FF_{q^3}/\FF_q)$, hence the subspace $\langle P, P^{\sigma}$, $P^{\sigma^2}\rangle \cap\Pi$ has the same dimension over $\FF_q$ and it has points of $\pi_N \subset \cV_1$. Suppose that all the planes of $\Sigma_{1,2}$ induce lines of $\Pi(\FF_{q^3})$, hence every plane of $\Sigma_{1,2}$ has a line that induces a point of weight 2 in $\Sigma(\FF_{q^6})$. The same is true for the planes of $\Sigma_{1,2}^{\sigma}$
and $\Sigma_{1,2}^{\sigma^2}$. If these lines span $\Pi(\FF_{q^3})$, we have that $\Pi(\FF_{q^3})$ is a plane of $\Sigma(\FF_{q^6})$ and hence $\Pi$ is a plane of $\Sigma$. Suppose that all these lines are contained in the same three-dimensional space, hence there exists a line $L$ of $\Sigma(\FF_{q^6})$ with $3(q^3+1)$ points of $\cV_1$, hence $L\subset \cV_1$. But $L$ is setwise fixed by $Gal(\FF_{q^6}/\FF_{q^2})$, hence there exists a subline of $L$ induced by $\Pi$ contained in $\cV_1$.
The last possibility is that $W$ is a variety of PG$(4,q^3)$, hence a hypersurface of PG$(4,q^3)$ of degree $3$. In this case, either one of $V(f)$ and $V(g)$ contains a hyperplane or there is a
hyperplane $H$ of PG$(5,q^3)$ such that $f\equiv g$ on $H$. In both cases, there exist $\alpha,\beta \in \FF_q$ not both equal to zero such that $\alpha f+\beta g=h^{1+\sigma+\sigma^2}$,
where $h$ is the linear polynomial defining $H$. Let $f=x_1y_1z_1+x_1y_2z_2+x_2y_1z_2+x_2y_2z_1+a x_2y_2z_2+f^*$ and
$g=x_1y_1z_2+x_1y_2z_1+x_2y_1z_1+a(x_1y_2z_2+x_2y_1z_2+x_2y_2z_1)+(a^2+1)x_2y_2z_2+g^*$, so $f^*$ and $g^*$ contain monomials with at least one variable raised to the second power. Any linear
combination of $f$ and $g$ cannot contain the monomials $x_1x_2y_1,x_1x_2y_2,x_1x_2z_1,x_1x_2z_2$.  Let $h=a_1x_1+b_1y_1+c_1z_1+a_2x_2+b_2y_2+c_2z_2$, then the coefficient of $x_1x_2y_1$
is $Tr(a_1a_2^qb_1^{q^2}+a_1a_2^{q^2}b_1^{q})$, where $Tr:\FF_{q^3}\rightarrow\FF_q$ is the usual trace function.  We have that, in even characteristic,
$Tr(a_1a_2^qb_1^{q^2}+a_1a_2^{q^2}b_1^{q})=\left|
\begin{array}{ccc}
a_1 & a_2 & b_1 \\
a_1^q & a_2^q & b_1^q \\
a_1^{q^2}& a_2^{q^2} & b_1^{q^2} \\
\end{array}
\right|$ and, by \cite[Lemma 3.51]{finitefields}, this matrix is singular if and only if $\{a_1,a_2,b_1\}$ is dependent over $\FF_q$. Analogously, by the lack of $x_1x_2y_2,x_1x_2z_1,x_1x_2z_2$, we get that also $\{a_1,a_2,b_2\}$, $\{a_1,a_2,c_1\}$ and $\{a_1,a_2,c_2\}$ are dependent over $\FF_q$. So either $b_1,b_2,c_1,c_2 \in \langle a_1,a_2 \rangle_{\FF_q}$ or $\{a_1,a_2\}$ is dependent over $\FF_q$. In the first case, we would have $\dim_{\FF_q} \langle a_1,a_2,b_1,b_2,c_1,c_2 \rangle=2$, but this is not possible. In fact, any linear combination $\alpha f + \beta g$ with $(\alpha,\beta)\neq (0,0)$ contains at least one monomial of type $x_iy_jz_h$, as $f+f^*$ contains monomials not contained in $g+g^*$ and vice versa, and this implies that $\dim_{\FF_q}\langle a_i,b_j,c_h\rangle=3$, a contradiction to $\dim_{\FF_q} \langle a_1,a_2,b_1,b_2,c_1,c_2\rangle =2$. So $\{a_1,a_2\}$ is dependent over $\FF_q$. Analogously, the lack of $y_1y_2x_1,y_1y_2x_2,y_1y_2z_1,y_1y_2z_2$ implies that $\{b_1,b_2\}$ is dependent over $\FF_q$ and the lack of $z_1z_2x_1,z_1z_2x_2,z_1z_2y_1,z_1z_2y_2$ implies that $\{c_1,c_2\}$ is dependent over $\FF_q$. So $a_2=\lambda a_1$, $b_2=\mu b_1$ and $c_2=\nu c_1$, with $\lambda,\mu,\nu \in \FF_q$, and $T:=Tr(a_1b_1^qc_1^{q^2}+a_1b_1^{q^2}c_1^{q})\neq 0$. The coefficient of $x_1y_1z_1$ in $\alpha f +\beta g$ is $\alpha$, while it is $T$ in $h^{1+\sigma+\sigma^2}$, hence $\alpha=T$. The coefficient of $x_1y_1z_2,x_1y_2z_1$ and $x_2y_1z_1$ in $\alpha f +\beta g$ is $\beta$, while they are, respectively, $\lambda T,\mu T, \nu T$ in $h^{1+\sigma+\sigma^2}$, hence $\lambda=\mu=\nu$ and $\lambda T=\beta$. Finally, the coefficient of $x_1y_2z_2$ is $\alpha + a \beta$ in $\alpha f +\beta g$ and it is $\mu\nu T=\lambda^2 T$ in $h^{1+\sigma+\sigma^2}$, so we get $1+a\lambda +\lambda^2=0$, but $t^2+at+1$ is irreducible over $\FF_q$ by hypothesis.
\end{proof}

\begin{theorem}
Let $\hat{G}$ be the group of collineations fixing $\cV$ (and hence $\cV_1$). The planes disjoint from $\cV_1$ form a unique orbit under the action of $\hat{G}$, and hence such a orbit consists of the planes inducing the Desarguesian spread.
\end{theorem}
By Proposition \ref{empty}, a plane $\pi$ is disjoint from $\cV_1$ if and only if $\pi\cap \cV_1$ consists of three nonconcurrent lines through three points $P,P^{\sigma},P^{\sigma^2}$ of $\cV$ over $\FF_{q^3}$, with $\sigma$ the $\FF_q$--linear collineation induced by $Gal(\FF_{q^6}/\FF_{q^2})$. We have that $P^{\sigma^i}=v_2(R^{\sigma^i})$, $R \in$PG$(2,q^6)$, $i=0,1,2$, where by abuse of notation we have denoted by $\sigma$ also the $\FF_{q^2}$--linear collineation of PG$(2,q^6)$ induced by $Gal(\FF_{q^6}/\FF_{q^2})$. If the points $R,R^{\sigma},R^{\sigma^2}$ were collinear, then $\pi$ would contain the image of the line through them, i.e. $\pi$ would contain a conic of $\cV$ and hence $\pi\subset \cV_1$. Hence $R,R^{\sigma},R^{\sigma^2}$ are not collinear. Let $R=(x,y,z)\in$PG$(2,q^6)$, then $R,R^{\sigma},R^{\sigma^2}$ are not collinear if and only if $\left(
                  \begin{array}{ccc}
                    x & y & z \\
                    x^{q^2} & y^{q^2} & z^{q^2} \\
                    x^{q^4} & y^{q^4} & z^{q^4} \\
                  \end{array}
                \right)
$ is nonsingular and by \cite[Lemma 3.51]{finitefields} this is equivalent to having $\{x,y,z\}$ independent over $\FF_{q^2}$.
 Let $R'=(x',y',z')\in$PG$(2,q^6)$ be another point such that $R',R'^{\sigma},R'^{\sigma^2}$ are not collinear, hence $\{x,y,z\}$ and $\{x',y',z'\}$ are two bases of $\FF_{q^6}$ considered as $\FF_{q^2}$-vector space, so there exists an element $g\in G=PGL(3,q^2)$ such that $\{x,y,z\}^g=\{x',y',z'\}$. As $g$ and $\sigma$ commute, $\forall g \in G$, we have that $G$ is transitive on the sets $\{R,R^{\sigma},R^{\sigma^2}\}$ with $R\in$PG$(2,q^6)$ such that $\{R,R^{\sigma},R^{\sigma^2}\}$ is not contained in a line. This implies that the lifting of $G$, say $\hat{G}$, fixing $\cV$ acts transitively on the planes $\pi$ of PG$(5,q^2)$ disjoint from $\cV_1$. This unique orbit hence contains the plane inducing the Desarguesian spread.

\end{document}